\documentclass[12pt, a4paper]{amsart} 
\usepackage[T2A]{fontenc}
\usepackage[utf8]{inputenc}
\usepackage[dvips]{graphicx}
\usepackage{geometry}
\usepackage{amsmath, amssymb, amsfonts, amsthm}
\usepackage{enumitem}
\usepackage{graphicx}
\usepackage{tikz}
\usepackage{bbm}

\usepackage{comment}
\usepackage[utf8]{inputenc}
\usepackage[T2A]{fontenc}
\usepackage[english]{babel}
\usepackage[foot]{amsaddr}

\newtheorem{theorem}{Theorem}
\newtheorem{theom}{Theorem}

\newtheorem{lemma}{Lemma}
\newtheorem{prop}{Proposition}
\newtheorem{defin}{Definition}
\theoremstyle{remark}
\newtheorem*{rem}{Remark}

\newenvironment{proof}{\noindent{\bf Proof:}}{$\hfill \Box$ \vspace{10pt}}  
\def\Xint#1{\mathchoice
{\XXint\displaystyle\textstyle{#1}}%
{\XXint\textstyle\scriptstyle{#1}}%
{\XXint\scriptstyle\scriptscriptstyle{#1}}%
{\XXint\scriptscriptstyle\scriptscriptstyle{#1}}%
\!\int}
\def\XXint#1#2#3{{\setbox0=\hbox{$#1{#2#3}{\int}$ }
\vcenter{\hbox{$#2#3$ }}\kern-.6\wd0}}

\def\dashint{\Xint-}
\newcommand{\hel} {
\hskip2.5pt{\vrule height7pt width.5pt depth0pt}
\hskip-.2pt\vbox{\hrule height.5pt width7pt depth0pt}
\, }
\newcommand{\restr}{\hel}

\begin{document}

\title[The Beurling and Malliavin Theorem
 in several dimensions]{The  Beurling and Malliavin Theorem in several dimensions}
\author[Ioann Vasilyev]{Ioann Vasilyev}
\address{Laboratoire Analyse, G\'eom\'etrie et Mod\'elisation, CY Cergy Paris Universit\'e, 2 avenue Adolphe
Chauvin, 95300, Pontoise, France
\newline
\textit{On leave from: St.-Petersburg Department of V.A. Steklov Mathematical Institute, Russian Academy of Sciences (PDMI RAS), Fontanka 27, St.-Petersburg, 191023, Russia}
}
\email{ioann.vasilyev@cyu.fr, ioann.vasilyev239@gmail.com}

\begin{abstract}
The present paper is devoted to a new multidimensional generalization of the First Beurling and Malliavin  Theorem, which is a classical result in the Uncertainty Principle in Fourier Analysis.  In more detail, we establish a new sufficient condition for a radial function to be a Beurling and Malliavin majorant in several dimensions (this means that the function in question can be minorized by the modulus of a square integrable function which is not zero identically and which has the support of the Fourier transform included in an arbitrary small ball). As a corollary of the radial case, we also obtain a new sharp sufficient condition in the nonradial case. The latter result provides a partial answer to the question posed by L. H\"ormander in the paper~\cite{hor}. Our proof is different in the cases of odd and even dimensions. In the even dimensional case we make use of  one classical formula from the theory of Bessel functions due to N. Ya. Sonin.
\end{abstract}

\maketitle

\section{Acknowledgments}
The author is kindly grateful to Alex Cohen and Semyon Dyatlov for valuable comments on a previous version of this paper and to  Sergei Kislyakov, Alexei Kulikov  and Mikhail Vasilyev 
 for helpful discussions and suggestions that helped to improve the manuscript. 

\section{Introduction}
The following classical theorem is called the First Beurling and Malliavin Theorem, see~\cite{beumal}.

\begin{theom}(A. Beurling and P. Malliavin)
\label{thmBM}
	Let $\omega: \mathbb R\rightarrow (0,1]$ be a continuous function such that $\log(1/\omega) \in L^1 (\mathbb R, (1+x^2)^{-1}dx),$ with $\log(1/\omega)$ satisfying a Lipschitz condition. Then, for each $\delta > 0$ there exists a function $f \in L^2 (\mathbb R),$ not identically zero, such that $\mathrm{supp}(\widehat f) \subset [0,\delta]$ and $|f(x)| \leqslant \omega (x)$ for all $x \in \mathbb R.$
\end{theom}

Let us make a few remarks about Theorem~\ref{thmBM} here.

\begin{rem}
If the claim of Theorem~\ref{thmBM} holds for some function, then we shall call this function a Beurling and Malliavin majorant.
\end{rem}

\begin{rem}
A square-integrable function whose Fourier transform has compact support will be further called bandlimited.
\end{rem}

This result is among the most spectacular results in Analysis obtained in the XXth century. For its significance in Harmonic Analysis, see the books~\cite{koos1} and~\cite{koos2} and the paper~\cite{nazhav}. We shall only mention that this theorem is a crucial tool in the proof of the Second Beurling and Malliavin Theorem about the radius of completeness of an exponential system, see e.g.~\cite{polt},~\cite{makpol}.
One can find our recent results connected with Theorem~\ref{thmBM} in  papers~\cite{vasil} and~\cite{vasil1}.

Let us give a very brief history of the problem since Beurling and Malliavin. There is a different variant of the First Beurling and Malliavin Theorem, namely the one dealing with weights of the form $|f|$, where $f$ is an entire function of Cartwright class. Note that Beurling and Malliavin in their paper of $1962$ proved both variants using the same idea, while Koosis clarified later that both variants are actually equivalent. 

Theorem~\ref{thmBM} is of great interest by itself since it can be thought of as a limit of applicability of the following well known and very important maxim called the Uncertainty Principle in Harmonic Analysis : ``It is impossible for a nonzero function and its Fourier transform to be simultaneously very small, unless the function is zero''. Indeed, Theorem~\ref{thmBM} indicates that there are nonzero functions that are ``small'' together with their Fourier transform.

To the best of our knowledge, one can not find an easy proof of Theorem~\ref{thmBM} since all known proofs of this result build upon certain complicated nonlinear constructions.

Note that even an easy consequence of Theorem~\ref{thmBM}, where the Lipschitz regularity of  the function $\log(1/\omega)$ is replaced by its decay on the negative ray and its growth on the positive one has important applications  for exponential systems. The main sources for this discussion are~\cite{redh} and~\cite{march}.

We would like to remind to the reader that the one-dimensional First Beurling and Malliavin Theorem (we shall further sometimes write ``BM Theorem'' instead of ``Beurling and Malliavin Theorem'' to save space) in the formulation of  Theorem~\ref{thmBM} was used by Ph. Tchami\-tchian for Beurling algebras, see~\cite{tcham}, by A. Beurling for the logarithmic transform of charges, see~\cite{havjor} and by J. Bourgain and S. Dyatlov for the fractal uncertainty principle, see~\cite{boudya}. 

\begin{rem}
For more on the fractal uncertainty principle, see papers~\cite{schlhan},~\cite{jami} and~\cite{jinzhang}. Note that in all these works the authors use Theorem~\ref{thmBM} (or its closely related variant to be revealed in a moment in Theorem~\ref{hanshlagonedim} of this paper) to prove the corresponding fractal uncertainty principle.   
\end{rem}

\begin{rem}
The condition  $\log(1/\omega) \in L^1 (\mathbb R, (1+x^2)^{-1}dx)$ imposed  on the function $\log(1/\omega)$ in Theorem~\ref{thmBM}  is the ``essential one'', since it controls the smallness of $\omega$, whereas the Lipschitz continuity is an auxiliary though very important regularity assumption. 
\end{rem}


There is a variant of Theorem~\ref{thmBM} due to V. Havin and J. Mashreghi  that is slightly easier to prove, see~\cite{havmas1},~\cite{barhav} and~\cite{belhav}. However, it can be treated as a replacement of Theorem~\ref{thmBM} in many cases. Recall a quantitative variant of the result by Havin and Mashreghi, see~\cite{schlhan}, Theorem B4.
\begin{theom} (R. Han and W. Schlag)
	\label{hanshlagonedim}
	Assume that $\omega: \mathbb R\rightarrow (0,1)$ satisfies $\log(1/\omega) \in L^1(\mathbb R, (1+x^2)^{-1}dx)$ and $\|(\mathcal H\log(1/\omega))^\prime\|_\infty \leqslant \pi\sigma/2$,
	for some $0<\sigma<1/10$, where $\mathcal H$ is the Hilbert transform on the real line. Then, there exists $\psi \in L^2(\mathbb R)$ with 
$\mathrm{supp}(\psi) \subset [0,\sigma]$, $|\widehat{\psi}| \leqslant \omega$ and $C(\sigma) \omega \leqslant |\widehat{\psi}|$  on  $[-3/4,3/4]$, where $C(\sigma)$ is an explicit positive constant that  depends on $\sigma$ only.
\end{theom}
We would like to stress that the Hilbert transform used in Theorem~\ref{hanshlagonedim} is not the classical one. For a formal definition of the  Hilbert transform $\mathcal H$ from Theorem~\ref{hanshlagonedim}, see the section preliminaries below.

Theorem~\ref{hanshlagonedim} in comparison to Theorem~\ref{thmBM} gives a quantitative  lower bound on the constructed bandlimited function. According to\footnote{A. Borichev, personal communication}, there is no version of Theorem~\ref{thmBM} in the literature that has this property.  Moreover, Theorem~\ref{hanshlagonedim} implies an
explicit essential spectral gap for convex co-compact hyperbolic surfaces in the case when the Hausdorff dimension of their limit set is close to $1$, see e.g. paper~\cite{jinzhang}.

\bigskip

A natural question that can be posed is to find multidimensional versions of the Beurling and Malliavin Theorems.  
Indeed, according to\footnote{A. Poltoratski, personal communication}, L. H\"ormander was trying to find in the late $1960$s multidimensional analogues of the BM Theorems, without substantial progress, however. This is also mentioned in the paper~\cite{hor}\footnote{Indeed, let us cite H\"ormander in~\cite{hor}: ``Theorem 1.1 is also a consequence of theorems of Beurling and Malliavin when $n=1$. However, no analogue of these is known when $n>1$, \ldots \,.'' Here, H\"ormander clearly indicates both theorems in the $1962$ paper by Beurling and Malliavin.}.
We were able to prove the following result which is, in our opinion, the multidimensional analogue of the First BM Theorem. This result can be regarded as the main one in the present paper.

\begin{theorem}
\label{bermalrad}
Let $\omega: \mathbb R^d\rightarrow(0,1]$ be a radial  function  
such that 
\begin{itemize}
		\item[---] 
		$\log(1/\omega) \in L^1 (\mathbb R^d, (1+|x|^2)^{-(d+1)/2}dx)$,
		\item[---] 
		$\log(1/\omega)$ is a Lipschitz function. 
\end{itemize}	
		Then, for each $\sigma > 0$ there exists a function $f \in L^2 (\mathbb R^d),$ not identically zero, such that $\mathrm{supp}(\widehat f) \subset B(0,\sigma)$ and $|f| \leqslant \omega$ on $\mathbb R^d.$
\end{theorem}

\begin{rem}
Throughout the present paper, the usual Euclidean norm in $\mathbb R^d$ is denoted by  $|\cdot|$.
\end{rem}

Note that Theorem~\ref{bermalrad} is indeed a generalization of the First BM Theorem to several dimensions.

\bigskip
We shall also derive from Theorem~\ref{bermalrad}  the following general (i.e. not necessarily radial) multidimensional majorant theorem, which is our second main result in this paper.
\begin{theorem}
	\label{nonradial}
Let $\omega: \mathbb R^d\rightarrow(0,1]$ be a function 
such that 
\begin{itemize}
		\item[---] 
		$\log(1/\omega) \in L^1 (\mathbb R^d, (1+|x|)^{-\gamma}dx)$, for some $\gamma<d+1$, 
		\item[---] 
		$\log(1/\omega)$ is a Lipschitz function. 
\end{itemize}
		Then, for each $\sigma > 0$ there exists a function $f \in L^2 (\mathbb R^d),$ not identically zero, such that $\mathrm{supp}(\widehat f) \subset B(0,\sigma)$ and $|f| \leqslant \omega$ on $\mathbb R^d.$
\end{theorem}

Several remarks are in order. First, notice that the bigger the parameter $\gamma$ is, the harder it is to prove Theorem~\ref{nonradial}. Second, the multidimensional generalization of the First BM Theorem that has been recently proved by A. Cohen in~\cite{coh2} requires bounds on three derivatives of the function $\log(1/\omega)$,  so   
our regularity assumption on $\log(1/\omega)$ is less restrictive than that in~\cite{coh2}. Third, it is not difficult to prove that in the case when $2\leqslant d$, the bound $\gamma< d+1$ is sharp within the framework of the weighted Lebesgue spaces $L^1(\mathbb R^d, (1+|x|)^{-\gamma}dx)$.

\begin{rem}
Note that in the case when the dimension of the ambient space is two or higher, we observe here a certain new effect. Indeed, Theorem~\ref{nonradial} and the very end of the discussion in the penultimate paragraph tell us that if $2\leqslant d$, then in this result there is no obvious ``critical parameter'' $\gamma$ within the framework of  weighted Lebesgue spaces $L^1(\mathbb R^d, (1+|x|)^{-\gamma}dx)$, in contrast to the one-dimensional case where $\gamma=2$ plays the role of  such parameter.
\end{rem}

\bigskip
We shall also prove the following variant of Theorem~\ref{bermalrad}.
\begin{prop}
	\label{havmashrad}
	Let $0<\sigma<1/10$ and let $w:\mathbb R^d\rightarrow (0,1]$ be a radial function satisfying $w(\cdot)=\mathcal\phi(|\cdot|)$ for a function $\phi: \mathbb R_+\rightarrow (0,1]$. Suppose that 
	\begin{enumerate}[label=(\roman*)]
		\item \label{one} $\phi\in L^2(\mathbb R_+,(1+x)^{2d+2}dx),$
		\item \label{two} $\log(1/\phi)\in L^1(\mathbb R_+, (1+x^2)^{-1}dx),$
		\item \label{three} $\|(\mathcal H_+\log(1/\phi))^\prime\|_\infty\leqslant \pi \sigma.$
		\end{enumerate}  
	Here, $\mathcal H_+$ is the Hilbert transform on the positive half-line. Then, there exists a function $\psi : \mathbb R^d \rightarrow \mathbb C$   that satisfies $ C(\sigma,d)\leqslant \|\psi\|_{L^2(B(0,1))}$, $\psi \in L^2 (\mathbb R^d)$ and $\mathrm{supp} (\widehat{\psi}) \subset B(0, \sigma)$ and such that $|\psi| \leqslant w$ on $\mathbb R^d$, where the constant $C(\sigma,d)$ is explicit and depends on $\sigma$ and $d$ only.
	\end{prop}
	
\begin{rem}
In Theorems~\ref{bermalrad} and~\ref{nonradial} and in Proposition~\ref{havmashrad} the bandlimited functions that we construct are in fact radial.
\end{rem}

The formal definition of the ``half'' Hilbert transform $\mathcal H_+$ will be given in the section preliminaries below.

Proposition~\ref{havmashrad} can be thought of as a multidimensional generalization of Theorem~\ref{hanshlagonedim}. The advantage of this version of this theorem is that it provides an explicit lower bound on the constructed bandlimited function in several dimensions.

\bigskip
Our proofs of Theorem~\ref{bermalrad} and Proposition~\ref{havmashrad} are quite simple and not technically involved. However, the proof is harder in the case of even dimensions since it  uses one rare formula from the Bessel functions theory, which is called the second Sonine integral, see formula~\eqref{bessel3} below. Indeed, let us briefly describe the main steps of the method that we are using in this paper to prove Proposition~\ref{havmashrad}.
\begin{enumerate}
\item[---] First, we prove an analogue of the Hahn--Schlag Theorem~\ref{hanshlagonedim} for the positive semi-axis, where the role of the Fourier transform is played by the cosine transform. Here, we shall use the fact that a Lipschitz function that is different from zero at  one point (more specifically, at zero) in fact can not be an identical zero, with an explicit lower bound on the $L^2$ norm. This step corresponds to Lemma~\ref{lem1} below.
\item[---] Second, we take a sufficient number of derivatives of the function constructed at the previous step and use Lord Ray\-leigh’s formula to deduce  in the odd-dimensional case the existence of desired functions with small spectrum. This step corresponds to the first and the second cases in the proof of Proposition~\ref{havmashrad}  below. 
\item[---] Third, the even-dimensional case is then derived from the odd-dimensional one, via Sonin's integral formula. This step corresponds to the third and the fourth cases in the proof of Proposition~\ref{havmashrad} below.
\end{enumerate}
 
 We note one difficulty in the proof of Theorem~\ref{bermalrad}, in comparison with the proof of Proposition~\ref{havmashrad}. As mentioned above, in Theorem~\ref{bermalrad} a lower bound on the constructed bandlimited function $f$ is not available, so this function, a priori, may turn out to be equal to zero at zero. However, this obstacle is quite easy to overcome by dividing the function $f$ by $x^N$, where $N$ is the order of the zero of $f$ at zero.
 
 We would like to emphasize the analogy of the main steps of the proof of results of this paper with those of the classical proof of the Huygens--Fresnel principle for the wave equation in Partial Differential Equations. This principle is only true in odd dimensions, and in the case of even dimensions only its weak version is valid, see e.g.~\cite{dyazwo} for a modern approach. In addition, the classical proof the Huygens--Fresnel principle via the Hadamard Method of Descent in the case of even dimensions is based on that in the odd-dimensional case, see e.g.~\cite{berjo} (see also~\cite{berver} for far-reaching generalizations of the Huygens--Fresnel principle), exactly as in the method described in the penultimate paragraph.

\bigskip
We should finally mention a number of open questions concerning Theorem~\ref{bermalrad}.  The first question consists in determining of whether Theorem~\ref{bermalrad} can be generalized to more general majorants $\omega$. 
 The second question is that 
it would be interesting to check whether  the main result of this paper could be applied to the fractal uncertainty principle. The third question concerns another variant of the First BM Theorem, namely the one dealing with weights of the form $|f|$, where $f$ is a function in the Cartwright class. To our mind, it would be interesting to find out whether this variant holds in the case of several dimensions as well. Finally, we hope that our Theorem~\ref{bermalrad} will find applications to the quantum chaos and to spectral gaps of resolvents of hyperbolic manifolds. 

We also believe that the method that we have applied in this paper will be useful in other higher-dimensional harmonic analysis questions. 

The rest of the current paper is organized as follows. The second section is the  preliminaries, while the third one contains the proofs of Proposition~\ref{havmashrad} and Theorem~\ref{bermalrad}. In the fourth section, which is the Appendix, we deduce Theorem~\ref{nonradial}  from Theorem~\ref{bermalrad}.

\section{Preliminaries}

We accumulate in this section the list of the frequently used technical abbreviations and notations and some classical and/or well known results that we shall use later on.

Let $\mathrm{Lip}(\mathbb R^d)$ denote the space of Lipschitz functions in $\mathbb R^d$ (i.e. functions $f$ satisfying for all $x,y\in \mathbb R$ the following inequality: $|f(x)-f(y)|\leqslant C|x-y|$ with $C>0$ independent of $x,y$). By $\mathrm{Lip}(\mathbb R,\kappa)$ we shall denote all Lipschitz functions in $\mathbb R^d$ with the Lipschitz constant $\kappa$.

For a Borel set $\Omega$, its characteristic function is denoted here $\chi_\Omega$. The Euclidean ball centered at $x\in \mathbb R^d$ and of radius $r>0$ is denoted $B(x,r)$ throughout this paper.

Recall that the Poisson measure $dP$ on $\mathbb R$ is defined by the following formula
$$dP(x) := \frac{dx}{1+x^2}.$$
The corresponding weighted Lebesgue space $L^1(dP)$ is the space of all functions $f$ satisfying $\int_{\mathbb R} \big| f \big| dP < \infty.$
The expression $\int_{\mathbb R} \log(1/\omega)dP$ will be sometimes further referred to as the logarithmic integral of ${\omega}$. For two vectors $\xi\in \mathbb R^d$ and $x\in \mathbb R^d$ their scalar product will be denoted $\xi.x$. 

Throughout this paper the signs $\lesssim$ and $\gtrsim$ indicate that the left-hand (right-hand) part of an inequality is less than the right-hand (left-hand) part multiplied by a ``harmless'' constant, whose value may vary  from line to line. 

We would like to remind the reader of how one should modify the Cauchy kernel in order to extend the classical definition of the Hilbert transform up to the space $L^1(dP).$

\begin{defin}
	The \textit{Hilbert transform} of a function $f\in L^{1}(dP)$ is defined as the following principal value integral
	$$\mathcal H f(x):=\dashint_{\mathbb R}\Bigl(\frac{1}{x-t}+\frac{t}{t^2+1}\Bigr)f(t)dt.
	$$
	It is worth noting that the integral above converges for almost all $x\in\mathbb R$.
\end{defin}

For a function $f\in L^{1}(\mathbb R_+, dP)$  we shall denote the ``half'' Hilbert transform or the Hilbert transform on the half-line by the following formula valid for almost all nonnegative $x$
$$\mathcal H_+ f(x):=\dashint_{0}^\infty\frac{2x f(t)}{x^2-t^2}dt.
$$
It is easy to check that if $f\in L^1(dP)$ is an even function on the real line, then we have
\begin{equation}
\label{frenchfries}
\mathcal H f(x)=\mathcal H_+ (f\restr_{\mathbb R_+})(x)
\end{equation} 
for almost all nonnegative $x$.

For a function $f \in L^2(\mathbb R^d),$ by $\mathrm{spec}(f)$ we mean the spectrum of ${f}$\, i.e. the support of its Fourier transform, which in turn is defined for functions $f\in L^1(\mathbb R^d)$ and $\xi\in \mathbb R^d$ by
$$\widehat{f}(\xi):=\int_{\mathbb R^d} e^{-2\pi i \xi.x}f(x)dx,$$ with the usual extension via  Plancherel's theorem and continuity to $L^2(\mathbb R^d)$. Note that the spectrum of a function in $L^2(\mathbb R^d)$ is defined up to a set of Lebesgue measure zero. Let us also remark that here, the term ``not identically zero'' means ``not zero almost everywhere''. We shall further sometimes write just ``nonzero'' for brevity.

\bigskip

We are going to use several classical results from the theory of Bessel functions. Recall (see~\cite{graf}, Appendix B) that for an index $\alpha>-1/2$ the corresponding Bessel function of the first kind admits the following integral representation, called the Poisson representation
\begin{equation}
\label{bessel1}
J_\alpha (y) = \left(\frac{y}{2}\right)^\alpha\frac{1}{\Gamma(\alpha+1/2)\Gamma(1/2)} \int_{-1}^1 e^{iys}(1-s^2)^\alpha\frac{ds}{\sqrt{1-s^2}},
\end{equation}
valid for all nonnegative $y$. We use this representation as a definition of Bessel functions of the first kind within this range of $\alpha$ and $y$. Here, $\Gamma$ stands for the gamma function. 

Note that  in the case when $\alpha$ is a half-integer, a well known Lord Rayleigh's formula reads
\begin{equation}
\label{bessel2}
J_{n+1/2} (y) = (-1)^n \sqrt{\frac{2}{\pi}} y^{n+1/2} \Biggl(\frac{1}{y} \frac{d}{dy}  \Biggr)^n \Biggl(\frac{\sin (y)}{y}   \Biggr),
\end{equation}
for all $n \in \mathbb N$ and $y \in \mathbb R_+.$ Note that it can be easily proven by mathematical induction (see~\cite{watson}, Example 17.2.12)  that for all $y\in \mathbb R_+$ and all odd numbers $d$ it holds that
\begin{equation}
\label{bessel2.5}
y^{d/2}J_{d/2-1} (y) = c(d)\left(\cos(y) P_d(y) + \sin(y) Q_d(y)\right),
\end{equation}
where  $P_d$ and $Q_d$ are algebraic polynomials satisfying 
$$P_d(y)=p_1y+p_3y^3+ \ldots +p_{2d_1+1}y^{2d_1+1}$$
 and 
$$Q_d(y)=q_0+q_2y^2+ \ldots +q_{2d_2}y^{2d_2},$$
where $d_1<d/2$ and $d_2<d/2$ are natural numbers depending on $d$ only. Their exact values will not matter in what follows. What will play an important role for us is that the polynomial $P_d$ has only nonzero coefficients corresponding to the even powers of $y$ and that the polynomial $Q_d$ has only nonzero coefficients corresponding to the odd powers of $y$. 

Next, notice the following rate decay at infinity
\begin{equation}
\label{mokosh}
|J_\alpha(y)|\lesssim \frac{1}{y^{1/2}},
\end{equation}
valid for all $y\in \mathbb R_+$. For a proof of~\eqref{mokosh}, see~\cite{watson}. We remark right away that the bound~\eqref{mokosh} will serve us later to interchange integrals and to differentiate under integral signs.

Moreover, we shall need the following formula due to N. Ya. Sonine (see~\cite{sonin}, Section 37, page 38):
\begin{equation}
\label{bessel3}
J_{\nu-\mu-1} (y) = c(\nu,\mu) y^{\mu+1} \int_{1}^\infty \frac{J_\nu(sy)}{s^{\nu-1}}(s^2-1)^\mu ds,
\end{equation}
valid in particular for all $-1<\mu<\nu<2\mu+1/2$ and positive $y$. The right hand side of formula~\eqref{bessel3} is sometimes referred to as the second Sonine integral. If parameters $\nu$ and $\mu$ satisfy $\nu-\mu>3/2$, then the integral in~\eqref{bessel3} converges absolutely. Otherwise, we shall consider this integral as an improper one. In particular, for $\nu=d/2-1/2$ with a natural $d>2$ and $\mu=-1/2$ the formula~\eqref{bessel3} reads as follows
\begin{equation}
	\label{bessel4}
	J_{d/2-1} (y) = c(d) y^{1/2} \int_{1}^\infty \frac{J_{d/2-1/2}(sy)}{s^{d/2-3/2}}(s^2-1)^{-1/2} ds.
\end{equation}
The advantage of formula~\eqref{bessel4} is that the parameter $s$ in the integral there is bounded away from zero. This feature of formula~\eqref{bessel4} will be very important for us in what follows.

Finally, we shall use the following fact. For a radial function $f$ on $\mathbb R^d$, i.e. $f(\cdot)=\mathrm{f}(|\cdot|)$ with some  function $\mathrm{f}$ defined on the positive half-line that belongs to an appropriate integrability class, the definition of the Fourier transform given  above reduces to the following formula
\begin{equation}
\label{radialll}
\widehat{f} (\xi)
	= 
\int_{\mathbb R^d} e^{-2\pi i \xi. x} \mathrm{f} (|x|)\, dx 
		= \big| \xi \big|^{-d/2 +1} \int_0^\infty r^{d/2} \mathrm{f} (r) J_{d/2-1} (| \xi| r) dr.
\end{equation}
For a proof of the formula~\eqref{radialll}, see~\cite{graf}. See also~\cite{graftes} for important connected  results.

\section {Radial majorant theorem in several dimensions }
In this section, we shall first prove Proposition~\ref{havmashrad}, i.e. the radial majorant theorem in several dimensions that we need.
In more detail, we shall deduce it from the one dimensional Theorem~\ref{hanshlagonedim}. We hope that after that we shall have finished the proof of Proposition~\ref{havmashrad}, it will be clear how to modify the proof in order to deduce our radial multidimensional First BM  Theorem~\ref{bermalrad} from its one dimensional variant. We shall comment on the required modifications at the end of this section.

\begin{proof} (of Proposition~\ref{havmashrad}.)
In the proof that follows we shall treat separately the cases when the dimension of the ambient space is two, three, when it is odd and when it is even. 

Recall that we have $w(\cdot) = \phi (|\cdot|)$ for a function $\phi : \mathbb R_+ \rightarrow (0,1].$ At the same time, we have $\log(1/\phi) \in \mathrm{Lip}(\mathbb R_+)\cap L^1 (\mathbb R_+, dP)$.


Consider an auxiliary operator $T$, defined for an integrable function $\varkappa$ by the following formula
$$T \varkappa (y) := \int_0^\infty \varkappa (r) \cos(ry)dr, \;\; \text{for} \;\; y \in \mathbb R_+$$
and called the cosine transform. This definition easily generalizes to $L^2(\mathbb R_+)$ functions. Let us state a technical result concerning the  operator $T$ that we shall shortly need.
\begin{lemma}
\label{lem1}
Let $0<\sigma<1/10$ and let $\varphi:\mathbb R_+\rightarrow (0,1]$ be a function satisfying  
		$\log(1/\varphi)\in L^1(\mathbb R_+,dP)$ and
		$\|(\mathcal H_+\log(1/\varphi))^\prime\|_\infty\leqslant \pi \sigma$.
  Then, there exists a function $g : \mathbb R_+ \rightarrow \mathbb C$   that satisfies 
$C(\sigma,\varphi)\leqslant \|g\|_{L^2(0,1)}$, $g \in L^2 (\mathbb R_+)$  and $\mathrm{supp} (Tg) \subset [0, \sigma]$
and such that $|g| \leqslant \varphi$ on $\mathbb R_+$. Here, $C(\sigma,\varphi)$ is a positive constant that may depend on $\sigma$ and $\varphi$.
\end{lemma}
\begin{proof} (of Lemma~\ref{lem1}.)
So, we shall now concentrate on the proof of Lemma~\ref{lem1}. Consider the even extension of the function $\varphi$. By this we mean the even function $\varphi_{\mathrm{ev}} : \mathbb R \rightarrow (0,1]$  that coincides with $\varphi$ on $\mathbb R_+$. 
 At the same time, we have 
$\log(1/\varphi_{\mathrm{ev}}) \in  L^1(dP)$   and $\|(\mathcal H\log(1/\varphi_{\mathrm{ev}}))^\prime\|_\infty\leqslant \pi \sigma$, see~\eqref{frenchfries}.
 This means that we can apply Theorem~\ref{hanshlagonedim}  to the function $\varphi_{\mathrm{ev}}.$ Hence, there exists a function	$f : \mathbb R \rightarrow \mathbb C$ that is not zero identically and such that 
$|f| \leqslant \varphi_{\mathrm{ev}}$, $f \in L^2 (\mathbb R)$  and 
$\mathrm{spec}(f) \subset [0, \sigma]$.
 Moreover, we infer that the inequality $C(\sigma) \varphi(0) \leqslant |f(0)|$ is valid for some positive constant $C(\sigma)$ depending on $\sigma$ only.


Consider the following function $f_{\mathrm{sym}}(\cdot) := f(\cdot) + f(-(\cdot)).$ Then, we obviously have that $\mathrm{spec}(f_{\mathrm{sym}}) \subset [- \sigma, \sigma]$ and that for all $y \in \mathbb R$ it holds that $f_{\mathrm{sym}}(-y) = f_{\mathrm{sym}}(y)$. Furthermore, we have
$$\big|f_{\mathrm{sym}}(y)\big| \leqslant \big| f(y) \big|+ \big| f(-y) \big| \leqslant \varphi_{\mathrm{ev}} (y) +\varphi_{\mathrm{ev}} (-y) = 2 \varphi_{\mathrm{ev}} (y).$$
On top of that, we see that $f_{\mathrm{sym}}\in L^2(\mathbb R)$ and that $f_{\mathrm{sym}}$ is nonzero identically, thanks to the facts that  $f_{\mathrm{sym}}(0) \neq 0$ and that the function $f_{\mathrm{sym}}$ is bandlimited, hence Lipschitz continuous, see e.g.~\cite{nazhav}, Section 1.4. In more detail, $f_{\mathrm{sym}}$ obeys the estimate $\|f_{\mathrm{sym}}^\prime\|_{\infty}\lesssim \sigma^{3/2} \|f\|_2$.

We are now ready to define the function $g$ as follows: $g := (f_{\mathrm{sym}}/2)\hel_{\mathbb R+}.$ Note that for all $t > 0$ it holds that
\begin{equation} 
	\label{eq2}
	\begin{split}
		\widehat{f_{\mathrm{sym}}}(t) &= \int_{-\infty}^0 e^{-2\pi i ty} f_{\mathrm{sym}}(y)\,dy + \int_0^\infty e^{-2\pi it y} f_{\mathrm{sym}}(y)\,dy \\
		&=\int_0^{ \infty} \Bigl( e^{-2\pi it y} + e^{2\pi it y} \Bigr) f_{\mathrm{sym}}(y)\,dy \\
		&= 4\int_0^{ \infty} g (y) \cos (2\pi ty)\, dy = 4T g (2\pi t).
	\end{split}
\end{equation}
This signifies that if $\tau \in (2\pi\sigma, \infty),$ then $Tg(\tau) = 0.$ 

Next, $g$ is obviously in $L^2(\mathbb R_+)$ and is Lipschitz continuous. Since also $|g(0)|=|f(0)|$, we conclude that the bound $\phi(0)\lesssim |g(0)|$ is verified by $g$. Moreover, for all nonnegative $t$ we have
$$\big| g (t) \big| = \frac{{\big| f_{\mathrm{sym}}}(t) \big|}{2} \leqslant \varphi_{\mathrm{ev}}(t) = \varphi (t),$$
which finishes off the proof of Lemma~\ref{lem1}.

\end{proof}

Let us now show  that Proposition~\ref{havmashrad} follows from Lemma~\ref{lem1}. Fix $0<\sigma<1/10$ and apply Lemma~\ref{lem1} to the function $\phi.$ This yields a certain function $g$. Consider the radial extension of $g$ defined by $\psi (x) := g (|x|)$ for $x \in \mathbb R^d$.

 We claim that the function $\psi$ satisfies  the required properties. Indeed, to see that $\psi\in L^2(\mathbb R^d)$, 
 we estimate as follows
$$\int_{\mathbb R^d}|\psi(x)|^2 dx\leqslant \int_{\mathbb R^d} \phi^2(|x|) dx \lesssim\int_0^\infty \phi^2(r)r^{d-1}dr< \infty,$$
thanks to~\ref{one}. The inequality $C(\sigma,\phi)\leqslant \|\psi\|_{L^2(B(0,1))}$  follows easily from the bounds on $|g(0)|$ and $\|g^\prime\|_{\infty}$  proved in Lemma~\ref{lem1}.

\bigskip
In the rest of the proof of Proposition~\ref{havmashrad}, we shall show that the function  $\psi$ is bandlimited. In order to check that property of $\psi$, we shall compute its Fourier transform.

We would like to start with the three dimensional case.

\textbf{Case $1$. Dimension three.} Let $\xi \in \mathbb R^3$ be an arbitrary point. By virtue of~\eqref{bessel2} and~\eqref{radialll}, 
\begin{equation} 
\label{eq1}
\begin{split}
\widehat \psi (\xi)
 &=  
\int_{\mathbb R^3} e^{-2\pi i \xi. x} g (|x|) dx 
= \big| \xi \big|^{-1/2} \int_0^\infty r^{1/2} g (r) J_{1/2} (| \xi| r) r dr \\ 
&= \big| \xi \big|^{-1} \sqrt{\frac{2}{\pi}} \int_0^\infty g (r) \sin (| \xi| r) r dr.
\end{split}
\end{equation}

Now, if $\tau \in (\sigma, \infty),$ then $Tg(\tau) = 0.$ This means that for such $\tau$ we have the formula $(T g)^\prime (\tau) = 0,$ which in turn signifies that 
\begin{equation*}
\int_0^\infty g (r) \sin (\tau r) r dr =0.
\end{equation*}
The reason why we are allowed to differentiate under the integral sign in $Tg$ is that the inequality $|g|\leqslant \phi$ and the decay rate~\ref{one} imposed on the function $\phi$ yield $g\in L^1(\mathbb R_+,r dr)$.

Bearing in mind~\eqref{eq1}, we infer the property $\widehat \psi (\xi) = 0$ once $| \xi | >\sigma$. This allows us to conclude in the three dimensional case. 

\bigskip

\textbf{Case $2$. Odd dimensions.} Next, in the case of an arbitrary odd dimension $d$ we write out the expression for the Fourier transform of $\psi$ using formulas~\eqref{bessel2.5} and~\eqref{radialll}
\begin{equation*} 
	\begin{split}
		\widehat \psi (\xi)
		&= \big| \xi \big|^{-d/2 +1} \int_0^\infty r^{d/2} g (r) J_{d/2-1} (| \xi| r) dr\\
		&=c(d)\big| \xi \big|^{-d +1} \int_0^\infty g (r) \Bigl(\cos(|\xi|r)P_d(|\xi|r)+\sin(|\xi|r)Q_d(|\xi|r)\Bigr) dr.
		\end{split}
\end{equation*}

Taking sufficiently many derivatives of the function $Tg$, we can see that if $\tau \in (\sigma, \infty)$, then for all integers $m$ from $0$ to $(d-1)/2$ it holds that
\begin{equation}
\label{moments}
\int_0^\infty g (r) \cos(\tau r) r^{2m} dr =0 \; \text { and } \; \int_0^\infty g (r) \sin (\tau r) r^{2m+1} dr =0.
\end{equation}
Recalling that the polynomial $P_d$ has only nonzero coefficients corresponding to the even powers and that the polynomial $Q_d$ has only nonzero coefficients corresponding to the odd powers, using formulas~\eqref{moments} we infer that $\widehat \psi (\xi)=0$ for all $\xi$ satisfying $|\xi|>\sigma$.  We hence also conclude in the case of an arbitrary odd dimension.

 \bigskip 

Further, due to a technical reason, we prefer to treat the two dimensional case separately.

\textbf{Case $3$. Dimension two.}  In this case  we shall write out the Fourier transform of the function $\psi$ using the Sonine representation formula~\eqref{bessel4} for Bessel functions of the first kind
\begin{equation} 
	\label{eqlalalalalend}
	\begin{split}
		\widehat \psi (\xi)
		&= \int_0^\infty r g (r) J_{0} (| \xi| r) dr \\
		&=c\int_0^\infty r g (r) \lim_{R\rightarrow \infty}\int_1^R J_{1/2} (s |\xi| r) \frac{(|\xi|r)^{1/2}s^{1/2}}{(s^2-1)^{1/2}} ds dr\\
		&=c \lim_{R\rightarrow \infty}\int_1^R\int_0^\infty g (r)\sin(s|\xi|r) r dr \frac{ds}{(s^2-1)^{1/2}},
	\end{split}
\end{equation}
where $\xi\in \mathbb R^2$ is arbitrary. Here, we can exchange the limit with respect to $R$ and the integration in $r$, thanks to the dominant convergence theorem. The integrals in $s$ and in $r$ can be exchanged by Fubini's theorem. 

Now, as in the three dimensional case, we can   
  see that in~\eqref{eqlalalalalend}, in the last line, the inner integral with respect to $r$ is zero once $s>1$ and $|\xi|>\sigma$.  
Hence, we conclude in the case of dimension two.
 
 \bigskip

\textbf{Case $4$. Even dimensions bigger than two.} In the case of an arbitrary even dimension $d\geqslant 4$ we also write out a formula for the Fourier transform of $\psi$, but now using the Sonine representation formula~\eqref{bessel4} for Bessel functions of the first kind
\begin{equation} 
	\label{eqlalalalend}
	\begin{split}
		\widehat \psi (\xi)  
		&= \big| \xi \big|^{-d/2 +1} \int_0^\infty r^{d/2} g (r) J_{d/2-1} (| \xi| r) dr \\
		&=c(d)\big| \xi \big|^{-d/2 +1} \int_0^\infty r^{d/2} g (r) \int_1^\infty J_{d/2-1/2} (s |\xi| r) \frac{(|\xi|r)^{1/2}(s^2-1)^{-1/2}}{s^{d/2-3/2}} ds dr\\
		&=c(d)\big| \xi \big|^{-d/2 +3/2} \int_1^\infty \frac{(s^2-1)^{-1/2}}{s^{d/2-3/2}} \int_0^\infty r^{d/2+1/2} g (r)  J_{d/2-1/2} (s |\xi| r)  dr ds,
	\end{split}
\end{equation}
where $c(d)$ is an absolute constant depending on the dimension $d$ only. The integrals above are interchangeable, thanks to the Fubini theorem. In effect, we can apply the Fubini theorem, thanks to bound~\eqref{mokosh} and condition~\ref{one}. Notice that here we are using that $d\geqslant 4$, in order to reassure that the double integral in question converges.

We have already seen in the proof of the odd-dimensional case that $Tg(\tau)=0$ for all $\tau>\sigma$ implies
\begin{equation} 
	\label{vazhno}
\int_0^\infty r^{d/2+1/2} g(r) J_{d/2-1/2}(\tau r) dr =0 
\end{equation}
for all $\tau>\sigma$. Indeed, to see that formula~\eqref{vazhno} holds, it suffices to replace $d$ by $d+1$ (which is now an odd integer) in the computations carried out in the proof of Case $2$.

Now, we use  line~\eqref{vazhno} 
 to see that in~\eqref{eqlalalalend}, in the last line, the inner integral with respect to $r$ is zero once $s>1$ and $|\xi|>\sigma$.  
Hence, we conclude in the case of an arbitrary even dimension.

Thus, Proposition~\ref{havmashrad} is proved in full generality.

\end{proof}

Let us now clarify how to modify the proof of Proposition~\ref{havmashrad}, in order to obtain a proof of Theorem~\ref{bermalrad}. 

\begin{proof} (of Theorem~\ref{bermalrad}.)
The only two points in the proof that need to be explained are: 
\begin{itemize}
\item how to obtain in the context of Theorem~\ref{bermalrad} some (not necessarily explicit) lower bound on the bandlimited function, that will be constructed in a counterpart of Lemma~\ref{lem1} just below and 
\item how to deal with the condition~\ref{one}. 
\end{itemize}

The problem mentioned in the latter point can be easily overcome via considering the function $\min(\omega,(1+|\cdot|)^{-Q})$ (with $Q$ sufficiently large) instead of $\omega$. Obviously, this modified function verifies~\ref{one},~\ref{two} and~\ref{three}. It is also clearly majorized by $\omega$. 

In order to explain how to overcome the obstacle from the former point, we shall formulate the following lemma which is a counterpart of Lemma~\ref{lem1}.
\begin{lemma}
\label{lem2}
Let $\varphi:\mathbb R_+\rightarrow (0,1]$ be a function satisfying $\log(1/\varphi)\in L^1(\mathbb R_+,dP)$ and such that  $\log(1/\varphi)$ is Lipschitz continuous.  Then, 
for each positive number $\sigma$ there exists a nonzero function $g : \mathbb R_+ \rightarrow \mathbb C$   that satisfies 
 $g \in L^2 (\mathbb R_+)$  and $\mathrm{supp} (Tg) \subset [0, \sigma]$
and such that $|g| \leqslant \varphi$ on $\mathbb R_+$.
\end{lemma}
\begin{proof}
(of Lemma~\ref{lem2}.) Recall that the even extension $\varphi_{\mathrm{ev}}$ of the function $\varphi$ has the convergent logarithmic integral
 and satisfies the bound $\|(\log(1/\varphi_{\mathrm{ev}}))^\prime\|_\infty<\infty.$ This means that we can apply Theorem~\ref{thmBM}  to the function  $\varphi_{\mathrm{ev}}.$ Hence, there exists a function $f : \mathbb R \rightarrow \mathbb C$ that is not zero identically and such that 
$|f| \leqslant \varphi_{\mathrm{ev}}$, $f \in L^2 (\mathbb R)$  and 
$\mathrm{spec}(f) \subset [0, \sigma]$.

If $f(0)\neq 0$, then there is nothing to prove, since we can proceed exactly as in the proof of Lemma~\ref{lem1}. If $f(0)= 0$, then we shall modify the function $f$. Since $f$ is bandlimited, it is also real analytic, so we know that there exists a natural number $N$ such that the function $f_0(x):=f(x)/x^N$ defined for real $x$ satisfies $f_0(0)\neq 0$ and $f_0\in C^\infty(\mathbb R)$. 

We claim that $\mathrm{spec}(f_0)\subset[0,\sigma]$. We shall only prove this in the case when $N=1$, since the general case follows easily by induction. By the Fourier inversion theorem, we have for all real $x$ 
$$f(x)=\int_0^\sigma F(t)e^{2\pi i t x}dt,$$
for some function $F\in L^2([0,\sigma])$. Since $f(0)=0$, we have $\int_0^\sigma F=0$. Let $\Phi$ denote the antiderivative of $F$ defined by
$$\Phi(t):=\int_0^t F(x)dx.$$ 
Note that clearly $\Phi\in L^2([0,\sigma])$ (in fact, it is easy to show that the function $\Phi$ is even H\"older continuous) and that  $\Phi(\sigma)=\Phi(0)=0$ and hence for all real $x$ it holds that
$$
\frac{f(x)}{x}=\int_0^\sigma \frac{e^{2\pi i t x}}{x}d\Phi(t)=-2\pi i \int_0^\sigma \Phi(t) e^{2\pi i t x} dt,
$$
and the claim follows.

\bigskip
We infer that there exists $\varrho>0$ such that
$M:=\max_{|x|\leqslant \varrho}f_0(x)$ satisfies $0<M<\infty$. Suppose, as we can, that $\varphi_{\mathrm{ev}}(x)=1$ for $-\varrho\leqslant x \leqslant \varrho$; this is explained, for instance, at the beginning of Section 2.3.1 of the paper~\cite{nazhav}. We treat separately two cases according to the value of $M$. In each of those cases we shall replace $f_0$ by a certain other function.

In the first case, $\varrho^{-N}\leqslant M$ and we choose the function $f_1:=f_0/M$ instead of $f_0$. We see that for $-\varrho\leqslant x \leqslant \varrho$ one has $|f_1(x)|\leqslant 1 = \varphi_{\mathrm{ev}}(x)$, and for $x$ such that $\varrho\leqslant|x|$ we have the inequalities
$$|f_1(x)|\leqslant \frac{|f(x)|}{M|x|^N}\leqslant |f(x)| \leqslant \varphi_{\mathrm{ev}}(x).$$

The second case corresponds to $M\leqslant \varrho^{-N}$. In this case, we take the function $f_2:=\varrho^N f_0$ instead of $f_0$. For $-\varrho\leqslant x \leqslant \varrho$ we see that 
$$|f_2(x)|\leqslant \frac{|f(x)|\varrho^N}{|x|^N}\leqslant M\varrho^N\leqslant\varphi_{\mathrm{ev}}(x),$$ 
and for $x$ such that $\varrho\leqslant|x|$ one has
$$|f_2(x)|\leqslant \frac{|f(x)|\varrho^N}{|x|^N}\leqslant |f(x)|\leqslant \varphi_{\mathrm{ev}}(x).$$

Also, in both cases $f_1(0)\neq 0$ and $f_2(0)\neq 0$. Therefore, we can finish off the proof of Lemma~\ref{lem2} exactly as in the proof Lemma~\ref{lem1}.

\end{proof}

The rest of the proof of Theorem~\ref{bermalrad} follows closely that of Proposition~\ref{havmashrad} and is not detailed here.

\end{proof}

\section{Appendix}
Let us now show how to derive Theorem~\ref{nonradial} from Theorem~\ref{bermalrad}.

\begin{proof} (of Theorem~\ref{nonradial}.)  We shall majorize the function $\Omega:=\log(1/\omega)$ by a certain radial Lipschitz function. To this end, denote  sets $E_0:=B(0,1)$ and for $j\in \mathbb N$, $E_j:=B(0,2^{j+1})\backslash B(0,2^{j})$ and numbers $\lambda_j:=\max_{y\in \bar{E_j}} \Omega(y)$. Define an auxiliary function $\Omega_1$ as follows: if $x\in E_j$, then  $\Omega_1(x):=\lambda_j$. This function clearly majorizes $\Omega$.

By the same reasons as those mentioned in the proof of Lemma~\ref{lem2}, there is no loss of generality to assume that $\Omega(0)=0$. Hence, we infer the property $\Omega(x)\leqslant \mathrm{C} 2^j$ which is true for $x\in E_j$, where $\mathrm{C}$ is the Lipschitz constant of the function $\Omega$.

Since the function $\Omega$ is Lipschitz, we know that for all $x\in \mathbb R^d$ it holds that $\lambda_j-\mathrm{C}|x-x_j|\leqslant \Omega(x)$ for the same constant $\mathrm{C}$ as above and for  certain points $x_j\in \bar{E_j}$. According to the conditions imposed on $\Omega$, we have
\begin{equation}
\label{salazar}
\begin{split}
\infty>\int_{\mathbb R^d} \Omega(x)\frac{dx}{(1+|x|)^\gamma}&\gtrsim \sum_{j=0}^\infty\int_{E_j}\Omega(x) \frac{dx}{2^{j\gamma}} \\
&\geqslant \sum_{j=0}^\infty 2^{-j\gamma} \int_{B\left(x_j,(2\mathrm{C})^{-1}\lambda_j\right)\cap E_j} (\lambda_j/2)dx\gtrsim\sum_{j=0}^\infty 2^{-j\gamma} \lambda_j^{d+1}.
\end{split}
\end{equation}
We would like to stress that we have used  in the last estimate of line~\eqref{salazar} the inequality $\Omega(x)\leqslant \mathrm{C} 2^j$ valid for $x\in E_j$.

On the other hand,
\begin{equation}
\label{veles}
\int_{\mathbb R^d} \Omega_1(x)  \frac{dx}{(1+|x|^2)^{\frac{d+1}{2}}} \lesssim \sum_{j=0}^\infty \int_{E_j}\lambda_j \frac{dx}{2^{j(d+1)}} \lesssim \sum_{j=0}^\infty \lambda_j 2^{-j}.
\end{equation}

Denote $\beta:=\gamma/(d+1)$ and note that $\beta<1$. Thanks to the H\"older inequality, we may write out line~\eqref{veles} as follows
\begin{equation}
\label{perun}
\int_{\mathbb R^d} \Omega_1(x) \frac{dx}{(1+|x|^2)^{\frac{d+1}{2}}} \lesssim \sum_{j=0}^\infty \frac{\lambda_j}{2^{j\beta}} 2^{j(\beta-1)} \lesssim \left(\sum_{j=0}^\infty\frac{\lambda_j^{d+1}}{2^{j\gamma}}\right)^{\frac{1}{d+1}}<\infty.
\end{equation}
It is maybe worth noting  that we have used bound~\eqref{salazar} in the last inequality of line~\eqref{perun}.

The function $\Omega_1$ is not necessarily Lipschitz. However, it is easy to modify $\Omega_1$ so that it becomes Lipschitz and that the property~\eqref{perun} is still verified. Theorem~\ref{nonradial} hence is proved, by virtue of Theorem~\ref{bermalrad}.

\end{proof}

\section{Data Availability Statement}
No datasets were generated or analysed during the current study.

\section{Conflict of interest}
On behalf of all authors, the corresponding author states that there is no conflict of interest.

\section{Competing interests}
The author has no relevant financial or non-financial interests to disclose.

\section{Funding}
This research was supported by the Russian Science Foundation (grant No.~23-11-00171), https://rscf.ru/project/23-11-00171/

\renewcommand{\refname}{References}

\end{document}